\documentclass[fontsize=12pt]{article}

\usepackage[english]{babel} 
\usepackage{amsmath,amsfonts,amsthm} 
\usepackage[utf8]{inputenc}
\usepackage{float}
\usepackage{lipsum} 
\usepackage{blindtext}
\usepackage{graphicx} 
\usepackage{caption}
\usepackage[sc]{mathpazo} 
\usepackage[T1]{fontenc} 
\usepackage{microtype} 
\usepackage[hmarginratio=1:1,top=32mm,columnsep=20pt]{geometry} 
\usepackage{multicol} 
\usepackage{booktabs} 
\usepackage{float} 
\usepackage{hyperref} 
\usepackage{lettrine} 
\usepackage{paralist} 
\usepackage{abstract} 
\usepackage{titlesec} 
\usepackage{natbib, url}

\usepackage{times}
\usepackage[T1]{fontenc}
\usepackage{microtype}
\usepackage{multicol}
\usepackage{xcolor}
\usepackage{comment}
\usepackage{tikz}
\usetikzlibrary{intersections}
\usepackage{newtxmath}
\usepackage[utf8]{inputenc}
\usepackage{multirow}
\usepackage[scale=0.90]{sourcecodepro}

\usepackage{caption}
\usepackage{subcaption}

\renewcommand\thesection{\Roman{section}} 

\titleformat{\section}[block]{\Large\bfseries\filright}{\thesection.}{1em}{} 
\titleformat{\subsection}[hang]{\bfseries\filright}{\thesubsection}{1em}{} 

\usepackage{fancyhdr} 
\usepackage[bottom]{footmisc}

\newtheorem{theorem}{Theorem}
\newtheorem{lemma}{Lemma}

\newtheorem{proposition}{Proposition}
\theoremstyle{definition}

\newtheorem{assumpt}{Assumption}

\newcommand{\given}{\,|\,}
\newcommand{\T}{\top}

\newenvironment{algo-fig}[2]{
	\begin{figure}[t!]
		$\textbf{#1}(#2)$
		\vspace*{2pt}
		\hrule
	}{
	\end{figure}
}

\newenvironment{algo-args}{
	\vspace*{-2pt}
	\begin{itemize}\setlength{\itemsep}{-2pt}
	}{
	\end{itemize}\setlength{\itemsep}{0pt}
}

\newenvironment{algo-algo}{
	\vspace*{-4pt}
	\textit{Algorithm}:\vspace*{-4pt}
	\begin{enumerate}\setlength{\itemsep}{-2pt}
	}{
	\end{enumerate}\setlength{\itemsep}{0pt}
}

\newenvironment{algo-return}{
	\vspace*{-8pt}
	\textit{Return}:
	\vspace*{-4pt}
	\begin{itemize}\setlength{\itemsep}{-2pt}
	}{
	\end{itemize}\setlength{\itemsep}{0pt}
	\hrule
}

\graphicspath{{./pics/}}

\title{\vspace{-15mm}\fontsize{24pt}{10pt}\selectfont\textbf{
Fixed-Domain Asymptotics Under Vecchia's Approximation of Spatial Process Likelihoods}} 

\author{
\large
{\textsc{Lu Zhang}}\\
{\normalsize Division of Biostatistics, Department of Population and Public Health Sciences, University of Southern California}\\
{\normalsize \href{mailto:lzhang63@usc.edu}{lzhang63@usc.edu}}\\
\\
\large
{\textsc{Wenpin Tang}\thanks{The first and second authors have equal contributions to this paper.}}\\
{\normalsize Department of Industrial Engineering and Operations Research, Columbia University}\\
{\normalsize \href{mailto:wt2319@columbia.edu}{wt2319@columbia.edu}}
\\
\\
\large
{\textsc{Sudipto Banerjee}}\\
{\normalsize Department of Biostatistics, University of California, Los Angeles}\\
{\normalsize \href{mailto:sudipto@ucla.edu}{sudipto@ucla.edu}}\\
}
\providecommand{\keywords}[1]{\textbf{\textit{Keywords:}} #1}
\begin{document}
\maketitle 
\thispagestyle{fancy} 

\label{firstpage}

\begin{abstract}
	Statistical modeling for massive spatial data sets has generated a substantial literature on scalable spatial processes based upon 
	Vecchia's approximation.
	Vecchia's approximation for Gaussian process models enables fast evaluation of the likelihood by restricting dependencies at a location to its neighbors. We establish inferential properties of microergodic spatial covariance parameters within the paradigm of fixed-domain asymptotics when they are estimated using Vecchia's approximation. The conditions required to formally establish these properties are explored, theoretically and empirically, and the effectiveness of Vecchia's approximation is further corroborated from the standpoint of fixed-domain asymptotics. 
\end{abstract}

\newpage
\section{Introduction}
Geostatististical data are often modeled by treating observations as partial realizations of a spatial random field. We customarily model the random field $\{Y(s) : s \in {\cal D} \}$ over a bounded region $ {\cal D} \in \mathbb{R}^d$ as a Gaussian process (GP), denoted $Y(s)\sim GP(\mu_{\beta}(s),K_{\theta}(\cdot,\cdot))$, with mean $\mu_{\beta}(s)$ and covariance function $K_{\theta}(s,s') = \mbox{cov}(y(s_i),y(s_j))$. The probability law for a finite set $\chi = \{s_1,s_2,\ldots,s_n\}$ is given by $y\sim N(\mu_{\beta}, K_{\theta})$, where $y=(y(s_i))$ and $\mu_{\beta}=(\mu_{\beta}(s_i))$ are $n\times 1$ vectors with elements $y(s_i)$ and $\mu_{\beta}(s_i)$, respectively, and $K_{\theta} = (K_{\theta}(s_i,s_j))$ is the $n\times n$ spatial covariance matrix whose $(i,j)$th element is the value of the covariance function $K_{\theta}(s_i,s_j)$. We consider the widely employed stationary Mat\'ern covariance function \citep{Matern86, Stein99} given by
\begin{equation}\label{eq:matern}
	K_{\theta}(s,s') := \frac{\sigma^2 (\phi \|h\|)^{\nu}}{\Gamma(\nu) 2^{\nu - 1}} {\cal K}_{\nu}(\phi \|h\|), \quad \|h\| \ge 0\;,
\end{equation}
where $h=s-s'$, $\sigma^2 > 0$ is called the {\em partial sill} or spatial variance, $\phi > 0$ is the {scale or decay parameter}, $\nu > 0$ is a smoothness parameter, $\Gamma(\cdot)$ is the Gamma function, ${\cal K}_{\nu}(\cdot)$ is the modified Bessel function of order $\nu$ \cite[Section 10]{AS65} and $\theta = \{\sigma^2,\phi,\nu\}$. 
The spectral density corresponding to (\ref{eq:matern}), which we will need later, is 
\begin{equation}\label{eq:spden}
	f(u) = C \frac{\sigma^2 \phi^{2 \nu}}{(\phi^2 + u^2)^{\nu + \frac{d}{2}}} \quad \mbox{for some } C>0.
\end{equation}

Likelihood-based inference for $\theta$ will require matrix computations in the order of $\sim n^3$ floating point operations (flops) and can become impracticable when the number of spatial locations, $n$, is very large. Writing $Y_n = (y_1,y_2,\ldots,y_n)^\T$, where $y_i := y(s_i)$ for $i=1,2,\ldots,n$ are the $n$ sampled measurements, we write the joint density $p(Y_n\given \theta) := N(Y_n;\; \mu_{\beta}, K_{\theta})$ as
\begin{equation}\label{eq: full_likelihood}
	p(Y_n \given \theta) = p(y_1 ; \theta) \prod_{i = 2}^n p(y_i \given y_{(i - 1)}; \theta)\;,
\end{equation}
where $y_{(i)} = (y_1, \ldots y_{i})$. \cite{ve88} suggested a simple approximation to (\ref{eq: full_likelihood}) based upon the notion that it may not be critical to use all components of $y_{(i - 1)}$ in $p(y_i \given y_{(i - 1)} ;\theta)$. Instead, the joint density $p(Y_n\given \theta)$ in \eqref{eq: full_likelihood} is approximated by 
\begin{equation}\label{eq: vecchia_approx}
	\tilde{p}(Y_n\given \theta) = p(y_1 \given \theta) \prod_{i = 2}^n p(y_i \given S_{(i - 1)}; \theta)\;,
\end{equation}
where $S_{(i)}$ is a subvector of $y_{(i)}$ for $i = 1, \ldots, n$. The density $\tilde{p}(Y_n\given \theta)$ in \eqref{eq: vecchia_approx} is called Vecchia's approximation and can be regarded as a quasi- or composite likelihood \citep{Zhang12, eidsvikshaby2014, BL20}. Vecchia's approximation has attracted a remarkable amount of attention in recent times, already too vast to be comprehensively reviewed here \citep[see, e.g.,][]{stein2004approximating, datta16a, datta16b, guinness2018permutation, kf2020vecch, katzfuss2021, pbf2022}. Algorithmic developments in Bayesian and frequentist settings \cite[][]{finley2019efficient, zdb2019, kf2020vecch} have enabled scalability to massive data sets (with $n \sim 10^7$ locations) and (\ref{eq: vecchia_approx}) lies at the core of several methods that tackle ``big data'' problems in geospatial analysis \citep{sun2012geostatistics, banerjee2017high, heaton2019case}. 

The Vecchia approximation has recently garnered substantial attention in the spatial statistics literature as an edifice for building massively scalable Gaussian process models. While substantial methodological innovation has been generated by this approach, developing a theoretical understanding regarding the inference and identifiability of the spatial process parameters has remained largely unaddressed. This is because the Vecchia approximation distorts the stationarity of the parent process and, hence, the theoretical tractability of the spatial processes are lost. Our current approach is an original first attempt based upon \cite{Zhang12} to formally introduce methods that can study the asymptotic properties of inference from Vecchia's approximation. While a completely rigorous development is available only in the one-dimensional setting, we emphasize that the approach we develop is novel and should generate subsequent theoretical research in two-dimensions. Therefore, we limit the formal theory to one-dimension but present some insightful numerical experiments in two-dimensions to show that the inferential behavior secured over the real line will be expected to carry over to spatial domains.

Following the fixed-domain (infill) asymptotic paradigm for spatial inference \citep{Steinbook, zhang2005towards} we discuss inferential properties for the parameters in (\ref{eq:matern}). In this setting, \cite{Zhang04} showed that not all parameters in $\theta$ admit consistent maximum likelihood estimators from the full Gaussian likelihood in (\ref{eq: full_likelihood}) constructed with a stationary Mat\'ern covariance function, but certain \textit{microergodic} parameters are consistently estimated. \citet[Theorem~5]{DZM09} formally established the asymptotic distributions of these microergodic parameters. \cite{KS13} addressed jointly estimating the 
{decay} and the variance parameters in the Mat\'ern family and the effect of a prefixed 
{decay} on inference when having relatively small sample size. All of the aforementioned work has been undertaken using (\ref{eq: full_likelihood}). Here, we formally establish the inferential properties for the estimates of microergodic parameters obtained from Vecchia's approximate likelihood in (\ref{eq: vecchia_approx}). We build our work on a brief but insightful discussion in Section~10.5.3 of \cite{Zhang12}, regarding the inferential behaviour arising from (\ref{eq: vecchia_approx}). To the best of our knowledge such explorations have not hitherto been formally undertaken. Following the aforementioned works in spatial asymptotics, we will restrict attention to the infill or fixed domain setting and 
{focus on the inferential properties of the microergodic parameters for any given value of the smoothness parameter. More specifically, we explore the criteria for asymptotic normality for the maximum likelihood estimates of the microergodic parameters obtained from Vecchia's approximation. In this regard, our work follows the paradigm laid out in \cite{Zhang12} in that we can no longer assume that the conditioning set is bounded. 
	{We provide conditions under which the inference under the Vecchia approximations of the Mat\'ern process will be asymptotically equivalent to the full model.} This distinguishes our intended contribution from that in \cite{BL20}, where bounded conditional sets are exploited to establish consistency results for some selected values of the smoothness parameter. On the other hand, we show that a different set of conditions can yield a closed form asymptotic distribution for any given value of the smoothness parameter.} For the subsequent development, it suffices to assume that $\mu_{\beta}(s)=0$, i.e., the data has been detrended. Hence, we work with a zero-centered stationary Gaussian process with the Mat\'ern covariance function in (\ref{eq:matern}), a fixed smoothness parameter $\nu$ and with the sampling locations $\chi_n$ restricted to a bounded region. 

{
	The balance of this article is arranged as follows. One of our key results, Theorem~\ref{thm:cond_matern}, is presented in Section~\ref{sec: general_cond}, providing general criteria for asymptotic normality of maximum likelihood estimates of microergodic parameters obtained from Vecchia's approximation. In Section~\ref{sec: Matern_example}, we demonstrate that these general criteria are implied by a condition on the conditioning size which grows much slower than the sample size. We numerically check the conclusions for one-dimensional cases and extend the discussion for two-dimensional cases in Section~\ref{subsec: sim_1}. 
}



\section{Infill Asymptotics for Vecchia's approximation}\label{sec: general_cond}
\subsection{Microergodic parameters}\label{subsec: Identifiability}
Identifiability and consistent estimation of $\theta$ in (\ref{eq:matern}) relies upon the equivalence and orthogonality of Gaussian measures. Two probability measures $P_1$ and $P_2$ on a measurable space $(\Omega, {\cal F})$ are said to be \textit{equivalent}, denoted $P_1 \equiv P_2$, if they are absolutely continuous with respect to each other. Thus, $P_1 \equiv P_2$ implies that for all $A\in {\cal F}$, $P_1(A)=0$ if and only if $P_2(A)=0$. On the other hand, $P_1$ and $P_2$ are orthogonal, denoted $P_1 \perp P_2$, if there exists $A\in {\cal F}$ for which $P_1(A)=1$ and $P_2(A)=0$. While measures may be neither equivalent nor orthogonal, Gaussian measures are in general one or the other. For a Gaussian probability measure $P_{\theta}$ indexed by a set of parameters $\theta$ 
{and $\kappa$, a function of $\theta$, we say that $\kappa(\theta)$ is \textit{microergodic} if $\kappa(\theta_1) \neq \kappa(\theta_2)$ implies $P_{\theta_1} \perp P_{\theta_2}$} \citep[see, e.g.,][]{Stein99, Zhang12}. Two Gaussian probability measures defined by Mat\'ern covariance functions $K_{\theta_1}(h)$ and $K_{\theta_2}(h)$, where $\theta_1 = \{\sigma^2_1,\phi_1,\nu\}$ and $\theta_2=\{\sigma^2_2,\phi_2,\nu\}$ are equivalent if and only if $\sigma_1^2 \phi_1^{2 \nu} = \sigma_2^2 \phi_2^{2 \nu}$ \citep[Theorem~2 in ][]{Zhang04}. Consequently, one cannot consistently estimate $\sigma^2$ or $\phi$ \citep[Corollary~1 in][]{Zhang04} from full Gaussian process likelihood functions, $\sigma^2 \phi^{2\nu}$ is a microergodic parameter that can be consistently estimated. 

If the oracle (data generating) values of $\phi$ and $\sigma^2$ are $\phi_0$ and $\sigma^2_0$, respectively, then for any fixed value of the decay $\phi = \phi_1$, we know from \citet[Theorem~5]{DZM09} that
\begin{equation}\label{eq: asym_normal_c}
	\sqrt{n}(\hat{\sigma}_n^2\phi_1^{2\nu}  - \sigma_0^2\phi_0^{2\nu}) \overset{{\cal L}}{ \longrightarrow} N(0, 2( \sigma_0^2\phi_0^{2\nu})^2)\;, 
\end{equation}
where $\hat{\sigma}_n^2$ is the maximum likelihood estimator from the full likelihood (\ref{eq: full_likelihood}). 

\subsection{Parameter estimation}\label{subsec: parameter_estimation}
Let $\hat{\sigma}_{n,vecch}^2$ be the maximum likelihood estimate of the variance $\sigma^2$,
\begin{equation}\label{eq:vercchiaphi1}
	\hat{\sigma}_{n, vecch}^2 = \mbox{argmax}_{\sigma^2}\{ \tilde{p}(Y_n\given \phi_1, \sigma^2),\; \sigma^2 \in \mathbb{R}^{+}\} \;,
\end{equation}
where $\tilde{p}(\cdot)$ is the density (\ref{eq: vecchia_approx}). We develop the asymptotic equivalence of $\hat{\sigma}_{n, vecch}^2$  with $\hat{\sigma}_{n}^2$. To proceed further, we introduce some notations. Assume that the target process $y(s)\sim GP(0, K_{\theta}(\cdot))$, where $K_{\theta}(h)$ is defined in (\ref{eq:matern}) with a fixed $\nu$. Let $P_{j}$, $j=0,1$ denote probability measures for $y(s)\sim GP(0, K_{\theta_i})$ with  $\theta_j = \{\sigma_j^2, \phi_j, \nu\}$. Assume that $\sigma_1^2 = \sigma_0^2 \phi_0^{2\nu} / \phi_1^{2\nu}$ and let $E_j(\cdot)$ denote the expectation with respect to probability measure $P_j$, $j = 0, 1$. We define 
\begin{equation}\label{eq: useful_quantities}
	\begin{split}
		e_{0, j} &:= y_1, \;  \mu_{i, j} := E_j(y_i \given y_{(i - 1)}),\; e_{i, j} := y_i - \mu_{i , j},\; i = 2, \ldots, n\\
		\tilde{e}_{0, j} &:= y_1, \;  \tilde{\mu}_{i, j} := E_j(y_i \given S_{(i - 1)}), \;\tilde{e}_{i, j} := y_i - \tilde{\mu}_{i , j},\; i = 2, \ldots, n\; .
	\end{split}
\end{equation}
In Lemma~\ref{lemma: useful_expression} we derive a useful expression for $\hat{\sigma}_{n,vecch}^2$ using the quantities in \eqref{eq: useful_quantities}.    
\begin{lemma}\label{lemma: useful_expression}
	The estimate of $\sigma^2$ from Vecchia's likelihood approximation with fixed $\nu$, and $\phi = \phi_1$ can be expressed as
	\begin{equation}\label{eq: useful_expression}
		\hat{\sigma}^2_{n , vecch} 
		= \frac{\sigma_1^2}{n} \sum_{i = 1}^n \frac{\tilde{e}_{i, 1}^2}{E_1\tilde{e}_{i, 1}^2}\;.
	\end{equation}
\end{lemma}
\begin{proof}
	In Vecchia's approximation \eqref{eq: vecchia_approx} with fixed $\nu$, $\phi = \phi_1$ and unknown $\sigma^2$ in $K_{\theta}(\cdot)$, $p(y_i \given S_{(i-1)})$ is Gaussian with mean $\widetilde{\mu}_{i,1} =  \widetilde{\Sigma}_{i,1}^{12}( \widetilde{\Sigma}_{i,1}^{22})^{-1} S_{(i-1)}$ and variance $\widetilde{\Sigma}_{i,1}:= \widetilde{\Sigma}_{i,1}^{11} - \widetilde{\Sigma}_{i,1}^{12} (\widetilde{\Sigma}_{i,1}^{22})^{-1}\widetilde{\Sigma}_{i,1}^{21}$,
	where $\displaystyle \begin{pmatrix}
		\widetilde{\Sigma}_{i,1}^{11} & \widetilde{\Sigma}_{i,1}^{12} \\
		\widetilde{\Sigma}_{i,1}^{21} & \widetilde{\Sigma}_{i,1}^{22}\end{pmatrix}$ is the covariance matrix of $\displaystyle \begin{pmatrix} y_{i} \\ S_{(i-1)} \end{pmatrix}$ under $\tilde{p}(\cdot; \phi_1, \sigma^2)$. Since $\widetilde{\mu}_{i,1}$ does not depend on $\sigma^2$, and $\widetilde{\Sigma}_{i,1}$ can be expressed as $\sigma^2  \widetilde{\Sigma}^{\dagger}_{i,1}$, where $\widetilde{\Sigma}^{\dagger}_{i,1}$ does not depend upon $\sigma^2$, the conditional distributions $p(y_i \given S_{(i - 1)})$ under Vecchia's approximation is
	\[
	p(y_i \given S_{(i - 1)}) = \frac{1}{\sqrt{2\pi \sigma^2 \widetilde{\Sigma}^{\dagger}_{i,1}}} \exp\left( - \frac{\tilde{e}_{i, 1}^2}{2 \sigma^2 \widetilde{\Sigma}^{\dagger}_{i,1}}  \right).
	\]
	A direct computation of \eqref{eq:vercchiaphi1} with any fixed $\phi_1$ yields (\ref{eq: useful_expression}),
	where we have used the fact $E_1 \tilde{e}^2_{i,1} = \sigma_1^2 \widetilde{\Sigma}^{\dagger}_{i,1}$ on the right hand side of (\ref{eq: useful_expression}). 
\end{proof}

Our main result builds on the discussion in Section~10.5.3 of \cite{Zhang12} to establish the following theorem that explores the asymptotic distribution of $\hat{\sigma}_{n, vecch}^2$. 

\begin{theorem}\label{thm:cond_matern}
	Assume that either of the following conditions holds:
	\begin{equation}\label{eq: thm1_cond2}
		\sum_{i = 1}^n \frac{E_0(\tilde{e}_{i, 1} - e_{i, 0})^2}{E_1\tilde{e}_{i, 1}^2} = \mathcal{O}(1) \quad \mbox{and} \quad
		\sum_{i = 1}^n \left(\frac{E_0e_{i, 0}^2}{E_1\tilde{e}_{i, 1}^2} - 1 \right)^2 = \mathcal{O}(1).
	\end{equation}
	or
	\begin{equation}\label{eq: thm1_cond3}
		\sum_{i = 1}^n \frac{E_1(\tilde{e}_{i, 1} - e_{i, 0})^2}{E_0e_{i, 0}^2} = \mathcal{O}(1) \quad \mbox{and} \quad
		\sum_{i = 1}^n \left(\frac{E_1\tilde{e}_{i, 1}^2}{E_0e_{i, 0}^2} - 1 \right)^2 = \mathcal{O}(1). 
	\end{equation}
	Then
	\begin{equation}\label{eq: thm1_vecch_MLE_asym_distr}
		\sqrt{n}(\hat{\sigma}_{n, vecch}^2\phi_1^{2\nu}  - \sigma_0^2\phi_0^{2\nu}) \overset{\mathcal{L}}{ \longrightarrow} N(0, 2( \sigma_0^2\phi_0^{2\nu})^2)\;.
	\end{equation}
\end{theorem}
Before presenting the proof of Theorem~\ref{thm:cond_matern}, we state and prove the following lemma.    
\begin{lemma}\label{lemma: useful_expression_2}
	The assumptions in (\ref{eq: thm1_cond2}) 
	imply that
	\begin{equation} \label{eq: cond_to_proveA}
		E_0\left[ \sum_{i = 1}^{n} \frac{\tilde{e}_{i, 1}^2}{E_1\tilde{e}_{i, 1}^2} - \sum_{i = 1}^{n} \frac{e_{i, 0}^2}{E_0e_{i, 0}^2} \right] = o(\sqrt{n}),
	\end{equation}
\end{lemma}
\begin{proof}
	We first prove that \eqref{eq: thm1_cond2} implies (\ref{eq: cond_to_proveA}). Note that 
	\begin{align*}
		\Bigg| E_0\left[ \sum_{i = 1}^{n} \frac{\tilde{e}_{i, 1}^2}{E_1\tilde{e}_{i, 1}^2} - \sum_{i = 1}^{n} \frac{e_{i, 0}^2}{E_0e_{i, 0}^2} \right] \Bigg| & = \Bigg| \sum_{i = 1}^{n} \left( \frac{E_0( \tilde{e}_{i, 1} - e_{i, 0} + e_{i, 0} )^2}{E_1\tilde{e}_{i, 1}^2} - 1 
		\right) \Bigg| \\
		& = \Bigg| \sum_{i = 1}^n \frac{E_0(\tilde{e}_{i, 1} - e_{i, 0})^2}{E_1\tilde{e}_{i, 1}^2} +
		\sum_{i = 1}^n \left(\frac{E_0e_{i, 0}^2}{E_1\tilde{e}_{i, 1}^2} - 1 \right) \Bigg| \\
		& \le \sum_{i = 1}^n \frac{E_0(\tilde{e}_{i, 1} - e_{i, 0})^2}{E_1\tilde{e}_{i, 1}^2}  + \sum_{i = 1}^n \Bigg| \frac{E_0e_{i, 0}^2}{E_1\tilde{e}_{i, 1}^2} - 1 \Bigg|,
	\end{align*}
	where the second equality follows from the fact that $\tilde{e}_{i, 1} - e_{i, 0}$ and $e_{i,0}$ are independent under $P_0$.
	By the first condition in \eqref{eq: thm1_cond2}, we get $\sum_{i = 1}^n E_0(\tilde{e}_{i, 1} - e_{i, 0})^2/E_1\tilde{e}_{i, 1}^2 = \mathcal{O}(1) = o(\sqrt{n})$. 
	Fix $\varepsilon>0$.
	By the second condition in \eqref{eq: thm1_cond2}, there is 
	$M > 0$ such that $\sum_{i>M} \left(E_0e_{i, 0}^2/E_1\tilde{e}_{i, 1}^2 - 1\right)^2<\varepsilon$.
	So for $n>M$, we can use the Cauchy–Schwarz inequality to obtain
	\begin{align*}
		\sum_{i=1}^n \left| \frac{E_0e_{i, 0}^2}{E_1\tilde{e}_{i, 1}^2} - 1 \right| &=\sum_{i=1}^M
		\left| \frac{E_0e_{i, 0}^2}{E_1\tilde{e}_{i, 1}^2} - 1 \right|+\sum_{i=M+1}^n \left| \frac{E_0e_{i, 0}^2}{E_1\tilde{e}_{i, 1}^2} - 1 \right| \\
		& \le \sum_{i=1}^M \left| \frac{E_0e_{i, 0}^2}{E_1\tilde{e}_{i, 1}^2} - 1 \right|+\sqrt{(n-M)\varepsilon}\;.
	\end{align*} 
	Dividing both sides by $\sqrt{n}$ reveals that
	$\limsup_{n \rightarrow \infty} \frac1{\sqrt{n}} \sum_{i=1}^n 
	\left| E_0e_{i, 0}^2/E_1\tilde{e}_{i, 1}^2 - 1 \right|
	\leq \sqrt{\varepsilon}$.
	Since $\varepsilon>0$ is arbitrary, it follows that
	$\sum_{i = 1}^n \left| E_0e_{i, 0}^2/E_1\tilde{e}_{i, 1}^2 - 1 \right| = o(\sqrt{n})$ 
	and we obtain \eqref{eq: cond_to_proveA}. 
	
\end{proof}

We now present a proof of Theorem~\ref{thm:cond_matern}.
\begin{proof}[Proof of Theorem~\ref{thm:cond_matern}]
	Recall that $\sum_{i = 1}^{n}e_{i, 0}^2/(E_0e_{i, 0}^2) = Y_n^\top V_{n, 0}^{-1} Y_n$, 
	where $V_{n , 0}$ is the covariance matrix of $Y_n$ under $P_0$.  Using (\ref{eq: useful_expression}) that was derived in Lemma~\ref{lemma: useful_expression}, we obtain 
	\begin{equation}\label{eq:split2}
		\sqrt{n}( \hat{\sigma}_{n, vecch}^2/\sigma_1^2 - 1 ) = \frac{1}{\sqrt{n}} \left[ 
		\sum_{i = 1}^{n} \frac{\tilde{e}_{i, 1}^2}{E_1\tilde{e}_{i, 1}^2} - \sum_{i = 1}^{n} \frac{e_{i, 0}^2}{E_0e_{i, 0}^2}\right] + \sqrt{n}\left( \frac{1}{n} Y_n^\top V_{n , 0}^{-1}Y_n - 1 \right).
	\end{equation}
	By the central limit theorem, we have $\sqrt{n}\left( \frac{1}{n} Y_n^\top V_{n , 0}^{-1}Y_n - 1 \right) \overset{\mathcal{L}}{\longrightarrow} N(0, 2)$. 
	
	We next show that the condition \eqref{eq: thm1_cond2} implies that
	\begin{equation}\label{eq: cond_to_prove}
		\begin{split}
			\sum_{i = 1}^{n} \frac{\tilde{e}_{i, 1}^2}{E_1\tilde{e}_{i, 1}^2} - \sum_{i = 1}^{n} \frac{e_{i, 0}^2}{E_0e_{i, 0}^2} &=  o(\sqrt{n}). 
		\end{split}
	\end{equation}
	To prove this, it will be sufficient to show that
	\begin{equation} \label{eq: cond_to_proveB}
		\sum_{i = 1}^{n} \frac{\tilde{e}_{i, 1}^2}{E_1\tilde{e}_{i, 1}^2} - \sum_{i = 1}^{n} \frac{e_{i, 0}^2}{E_0e_{i, 0}^2} - E_0\left[ \sum_{i = 1}^{n} \frac{\tilde{e}_{i, 1}^2}{E_1\tilde{e}_{i, 1}^2} - \sum_{i = 1}^{n} \frac{e_{i, 0}^2}{E_0e_{i, 0}^2} \right] = O(1).
	\end{equation}
	The result in (\ref{eq: cond_to_prove}) will then follow from Lemma~\ref{lemma: useful_expression_2}. We turn to proving \eqref{eq: cond_to_proveB}. Our argument relies on the equivalence of Gaussian sequences.
	Let $\widetilde{P}_{1,n}$ be the probability distribution corresponding to $\tilde{p}(Y_n; \phi_1,\sigma_1^2)$, and 
	let $\rho_n := \tilde{p}(Y_n; \phi_1,\sigma_1^2) / p(Y_n ; \phi_0,\sigma_0^2)$ be the Radon-Nikodym derivative of $\widetilde{P}_{1,n}$ with respect to $P_0$ on the realization $Y_n$ for a given $n$.
	Write $\widetilde{P}_{1,\infty}$ (respectively, $P_{0,\infty}$) for the probability distribution corresponding to $\tilde{p}(\cdot; \phi_1,\sigma_1^2)$ (respectively, $p(\cdot ; \phi_0,\sigma_0^2)$) on the infinite sequence $(Y_1, Y_2, \ldots)$.
	By Kakutani's dichotomy, $\widetilde{P}_{1,\infty}$ and $P_{0,\infty}$ are either equivalent or mutually singular to each other.
	If $\widetilde{P}_{1,\infty}$ is equivalent to $P_{0,\infty}$, then $\mbox{lim}_{n \rightarrow \infty} \rho_n = \rho_{\infty}=: d\widetilde{P}_{1,\infty} /dP_{0,\infty}$ with $P_{0}$-probability $1$ \citep[see, e.g.,][Section III.2.1]{IR78}. Also, $\mathbb{P}_0(0 < \rho_{\infty} < \infty) = 1$ and $- \infty < E_0(\log \rho_\infty) < \infty$.
	As a consequence, 
	\begin{align*}
		\log \rho_n &= -\frac{1}{2} \log \frac{\det \widetilde{V}_{n,1}}{\det V_{n,0}} - \frac{1}{2}\left(\sum_{i = 1}^{n} \frac{\tilde{e}_{i, 1}^2}{E_1\tilde{e}_{i, 1}^2} - \sum_{i = 1}^{n} \frac{e_{i, 0}^2}{E_0e_{i, 0}^2} \right) = \mathcal{O}(1), \\
		E_0(\log \rho_n) & = -\frac{1}{2} \log \frac{\det \widetilde{V}_{n,1}}{\det V_{n,0}} - \frac{1}{2}E_0\left[ \sum_{i = 1}^{n} \frac{\tilde{e}_{i, 1}^2}{E_1\tilde{e}_{i, 1}^2} - \sum_{i = 1}^{n} \frac{e_{i, 0}^2}{E_0e_{i, 0}^2} \right] = \mathcal{O}(1),
	\end{align*}
	where $\widetilde{V}_{n,1}$ (respectively, $V_{n,0}$) is the covariance matrix of $Y_n$ under $\widetilde{P}_{1,\infty}$ (respectively, $P_{0}$). By taking the difference of the above two equations, we get \eqref{eq: cond_to_proveB} under the condition that $\widetilde{P}_{1,\infty}$ is equivalent to $P_{0,\infty}$. 
	Using Theorem 5, Section VII.6 of \cite{Sh96}, we can conclude that
	\begin{equation}
		\label{eq:equivPP}
		\begin{aligned}
			\widetilde{P}_{1,\infty} \mbox{ is equivalent to } P_{0,\infty} & \Longleftrightarrow \sum_{i = 1}^{\infty} \left[ \frac{E_0(\tilde{e}_{i, 1} - e_{i, 0})^2}{E_1\tilde{e}_{i, 1}^2} + \left(\frac{E_0e_{i, 0}^2}{E_1\tilde{e}_{i, 1}^2} - 1 \right)^2 \right] < \infty \\
			& \Longleftrightarrow \sum_{i = 1}^{\infty} \left[ \frac{E_1(\tilde{e}_{i, 1} - e_{i, 0})^2}{E_0e_{i, 0}^2} +\left(\frac{E_1\tilde{e}_{i, 1}^2}{E_0e_{i, 0}^2} - 1 \right)^2 \right] < \infty.
		\end{aligned}
	\end{equation}
	Since first equivalence in \eqref{eq:equivPP} is simply a reformulation of \eqref{eq: thm1_cond2}, we have established that the condition \eqref{eq: thm1_cond2} implies \eqref{eq: cond_to_proveB} and, hence, the result in \eqref{eq: cond_to_prove}. 
	
	The proof from condition~\eqref{eq: thm1_cond3} to \eqref{eq: cond_to_prove} is established following the proof from condition~\eqref{eq: thm1_cond2} to \eqref{eq: cond_to_prove}. We now break the quantity $\sum_{i = 1}^{n} \tilde{e}_{i, 1}^2/E_1\tilde{e}_{i, 1}^2 - \sum_{i = 1}^{n} e_{i, 0}^2/E_0e_{i, 0}^2$ in \eqref{eq: cond_to_prove} into 
	\begin{multline*}
		\sum_{i = 1}^{n} \frac{\tilde{e}_{i, 1}^2}{E_1\tilde{e}_{i, 1}^2} - \sum_{i = 1}^{n} \frac{e_{i, 0}^2}{E_0e_{i, 0}^2} - E_1\left[ \sum_{i = 1}^{n} \frac{\tilde{e}_{i, 1}^2}{E_1\tilde{e}_{i, 1}^2} - \sum_{i = 1}^{n} \frac{e_{i, 0}^2}{E_0e_{i, 0}^2} \right] \\ 
		+ E_1\left[ \sum_{i = 1}^{n} \frac{\tilde{e}_{i, 1}^2}{E_1\tilde{e}_{i, 1}^2} - \sum_{i = 1}^{n} \frac{e_{i, 0}^2}{E_0e_{i, 0}^2} \right] \;,
	\end{multline*}
	and replace the left hand-side of \eqref{eq: cond_to_proveB} and \eqref{eq: cond_to_proveA} by the two quantities respectively. Then the equivalence of $P_1$ and $P_0$ along with lemma~\ref{lemma: useful_expression_2} shows that the replaced \eqref{eq: cond_to_proveA} holds. The proof that \eqref{eq: thm1_cond3} implies the replaced \eqref{eq: cond_to_proveB} remains the same except that we now use the second equivalence in \eqref{eq:equivPP}. 
	Thus we complete the proof of Theorem~\ref{thm:cond_matern}.
\end{proof}

Turning to the connection between Theorem~\ref{thm:cond_matern} and predictive consistency of Vecchia's approximation in the sense of \citet[p.478]{KS13}, note that $e_{i,0}$ and $\tilde{e}_{i,1}$ are the predictive errors for $y_i$ under the full model with correct parameters and under (\ref{eq: vecchia_approx}) with possibly incorrectly specified parameters $(\phi, \sigma) = (\phi_1, \sigma_1)$. A consequence of \eqref{eq: thm1_cond2} or \eqref{eq: thm1_cond3} is that $\displaystyle {E_1\tilde{e}_{i,1}^2/E_0 e_{i,0}^2} \to 1$ as $i$ (and hence $n$) is large.
Hence, \eqref{eq: thm1_cond2} or \eqref{eq: thm1_cond3} implies asymptotic normality of estimates as well as predictive consistency. 

\section{Infill Asymptotics for Vecchia's Approximation on the line}\label{sec: Matern_example} 
It is possible to obtain further insights into Theorem~\ref{thm:cond_matern} when considering the asymptotic normality of $\hat{\sigma}_{n, vecch}$ for Mat\'ern models with observations on the real line. Whilst the conditions (\ref{eq: thm1_cond2})~and~(\ref{eq: thm1_cond3}) are, in general, analytically intractable due to the presence of $E_0 \tilde{e}^2_{i,1}$, we will show that \eqref{eq: thm1_cond3} holds for Mat\'ern models on $\mathbb{R}$.  

To simplify the presentation, we consider the fixed domain $\mathcal{D} = [0,1]$, and the sampled locations $\chi = \{i/n: 0 \le i \le n\}$. Denote $\delta = 1/n$ for the spacing of $\chi$, and $y_i = y(i \delta)$, $0 \le i \le n$ for the observations. We define $S_{(i)} = S_{(i)}[k]: = (y_i, y_{i-1}, \ldots, y_{i-k+1})$ for a positive integer $k$, where $S_{(i)}[k]$ is the vector of $k$ consecutive observations backward from $y_i$. 
{The integer $k$ is capped by $i$ since $S_{(i)}[k]$ is a subvector of $y_{(i)}$.}
\begin{assumpt}\label{assumpt1}
	Let $\mathcal{D} = [0,1]$, and $\chi = \{i \delta: 0 \le i \le n\}$ with $\delta = 1/n$.
	Then
	\begin{equation}\label{eq:miscomp}
		\sum_{i = 1} ^n \frac{E_1(e_{i,1} - e_{i,0})^2}{E_1 e_{i,1}^2} = \mathcal{O}(1).
	\end{equation}
\end{assumpt}

Before stating the main result on $\mathbb{R}$, we demonstrate why this assumption is reasonable.  We empirically investigate $\sum_{i = 1} ^n E_1(e_{i,1} - e_{i,0})^2/E_1 e_{i,1}^2$ for increasing values of $n$. Figure~\ref{fig: assump_check} plots the values of $\sum_{i = 1} ^n E_1(e_{i,1} - e_{i,0})^2/E_1 e_{i,1}^2$ with $\chi = \{i\delta: 0 \le i \le n\}$ for $\nu = 0.25, 0.5, 1.0, 1.5, 2.0$, and $n$ ranging from $100$ to $1200$. As $n$ increases, the plot tends to flatten as is suggested by the assumption. 
\begin{figure}[htb]
	\centering
	\includegraphics[width = 0.75\linewidth]{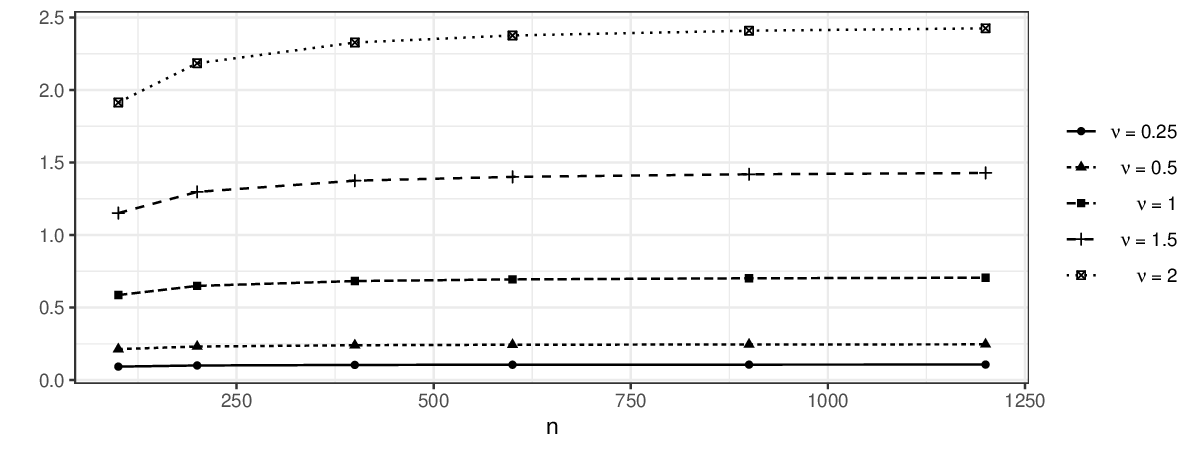}
	\caption{ Trend of $\sum_{i = 1} ^n E_1(e_{i,1} - e_{i,0})^2/E_1 e_{i,1}^2$ for Mat\'ern model when $\chi$ is a regular grid $\chi = \{i \delta: 0 \le i \le n\}$. Parameter $\sigma^2$ in Mat\'ern covariogram equals $1.0$ and decay $\phi$ for different $\nu$ are set to make the correlation of two points equals 0.05 when their distance reaches 0.2. \label{fig: assump_check}}
\end{figure}

Some additional explanation is also possible from a theoretical viewpoint. Since $P_0$ and $P_1$ are equivalent, Corollary 3.1 of \cite{Stein90} implies that 
$E_1(e_{i,1} - e_{i,0})^2 / E_1 e_{i,1}^2 \to 0$ as $n, i \to \infty$.
By stationarity and symmetry of the Mat\'ern model, 
$e_{i,j}$ is distributed as the error of the least square estimate of $y_0:= y(0)$ given observations $y_{(i)} = (y_1, \ldots, y_i)^{\T}$. Hence, $e_{i,j}$ can be realized as $y_0 - E_j(y_0 \given y_{(i)})$. Similarly, $\tilde{e}_{i,j}$ can be realized as $y_0 - E_j(y_0 \given S_{(i-1)})$.
Now consider the infinitely sampled locations $\{i \delta: i \ge 0\}$, and extend the finite sample $Y_n: = (y_0, \ldots, y_n)^{\T}$ to $Y: = (y_0, y_1, \ldots)^{\T}$ with $y_i : = y(i \delta)$. 
{The sampled locations of $Y$ form an infinite grid on $[0, \infty)$, and $\delta = 1/n$ is determined based on the sample size of $Y_n$.} Let $e_{\infty, j}$ be the error $y_0 - E_j(y_0 \given y_1, \ldots)$ for the infinite sequence $Y$.
For $f_0$ (resp. $f_1$) the spectral density under $P_0$ (resp. $P_1$), it is easily seen that $f_0, f_1 \sim C \sigma_0^2 \phi_0^{\nu} u^{-2 \nu - 1}$ as $u \rightarrow \infty$, and $(f_1 - f_0)/f_0 \asymp u^{-2}$. 
Therefore, by Theorem 2 of \cite{Stein99},
\begin{equation}
	\label{eq:miscomp2}
	\frac{E_1(e_{\infty,1} - e_{\infty,0})^2}{E_1 e_{\infty,1}^2} = O(\delta^{\min(2\nu+1,4)} \log(\delta^{-1})^{1(\nu = 3/2)}))\;.
\end{equation}
Intuitively, $e_{i,j} \approx e_{\infty, j}$ for large $i$. 
It will not be unreasonable to speculate a stronger result where \eqref{eq:miscomp2} still holds by replacing $e_{\infty, j}$ with $e_{i,j}$ for large $i$, which would imply \eqref{eq:miscomp}. As indicated on p.138 of \cite{Steinbook}, obtaining the rate of $E_1(e_{i,1} - e_{i,0})^2 / E_1 e_{i,1}^2$ for any bounded domain $\mathcal{D}$ is a highly non-trivial task. The only known results are obtained in \cite{Stein90b, Stein99} for $\nu = \frac{1}{2}, \frac{3}{2}, \ldots$, which also imply \eqref{eq:miscomp}.
In fact, we need that $E_1e_{i,1}^2 = E_1e_{\infty,1}^2 (1 + \mathcal{O}(i^{-\kappa}))$ for any $\kappa > 0$. Hence, the only missing piece in the above heuristics is an estimate of $E_1 e_{i,0}^2 / E_1 e_{\infty, 0}^2$, which we do not explore further here. Our result is stated as follows.
\begin{theorem}\label{thm:main2}
	Let $\mathcal{D} = [0,1]$, $\chi = \{i \delta: 0 \le i \le n\}$ and $S_{(i)} = S_{(i)}[n^{\epsilon}]$ for $\epsilon \in (0,1)$.
	If (\ref{eq:miscomp}) holds, then \eqref{eq: thm1_cond3}  also holds. Consequently, \eqref{eq: thm1_vecch_MLE_asym_distr} holds for the Mat\'ern model ($\nu > 0$).\end{theorem}

We make a few remarks before presenting a proof. Theorem~\ref{thm:main2} states that the asymptotic normality of the microergodic parameter $\sigma^2 \phi^{2 \nu}$ still holds under Vecchia's approximation in a neighborhood of 
{at most} size $k = n^{\epsilon} \ll n$ (sample size), where the computation of $\hat{\sigma}^2_{n, vecch}$ is much cheaper.
This justifies the validity of Vecchia's approximation for Mat\'ern models from a fixed-domain perspective. The range $n^{\epsilon}$ may not be optimal, and it might be possible to improve to $k = \mathcal{O}(\log n)$. However, we do not pursue this direction here from a theoretical standpoint. A simulation study is provided for the case of $k = \mathcal{O}(\log n)$ in Section~\ref{subsec: sim_1}.

An interesting situation arises with $\nu = 1/2$, in which the process reduces to the Ornstein-Uhlenbeck process, and 
$p(y_i \given y_{(i-1)}) = p(y_i \given y_{i-1})$. Therefore, \eqref{eq: thm1_vecch_MLE_asym_distr} trivially holds for Vechhia's approximation with a neighbor of size $k = 1 = \mathcal{O}(1)$. It is, therefore, natural to enquire whether the asymptotic normality of \eqref{eq: thm1_vecch_MLE_asym_distr} holds under Vecchia's approximation within a range $k = \mathcal{O}(1)$. 
If this is true, computational efforts can be reduced further. Unfortunately, this need not be the case. For $\nu < 1/4$ and $k = \mathcal{O}(1)$, $n^{2\nu}(\hat{\sigma}_{n, vecch}^2\phi_1^{2\nu}  - \sigma_0^2\phi_0^{2\nu})$ converges to a non-Gaussian distribution \citep{BL20}. The cases for $\nu \ge 1/4$ remain unresolved.

Now we turn to the proof of Theorem \ref{thm:main2}.
The key to this analysis is the following proposition, which relies on a result on the bound of $e_{\infty, j} - e_{i,j}$, i.e.  the difference between the errors of the finite and the infinite least square estimates \citep{Baxter62}.
The study dates back to the work of \cite{Kol41}, 
see also \cite{GS58, Ibr64, DM70, DM76, Gin99} for related discussions.

\begin{proposition}
		\label{lem:key}
		Let $\kappa > 0$. There exist $C_0, C_{\kappa} >0$ such that for $\delta < 1$ and $j = 0,1$,
		\begin{equation}
			\label{eq:inferror}
			E_j e_{\infty, j}^2 \sim C_0 \delta^{2 \nu},
		\end{equation}
		\begin{equation}
			\label{eq:differror}
			E_j (e_{\infty, j} - e_{i,j})^2 \le C_{\kappa} \delta^{2\nu} i^{-\kappa}.
		\end{equation}
	\end{proposition}
	\begin{proof}
		From the discussion below \eqref{eq:miscomp} that $e_{i,j}$ and $e_{\infty, j}$ can  be realized as $e_{i,j} = y_0 - E_j(y_0 \given y_1, \ldots, y_i)$ and $e_{\infty,j} = y_0 - E_j(y_0 \given y_1, y_2, \ldots)$,
		where $y_i: = y(i \delta)$ is indexed by nonnegative integers, we know from \citet[p.77]{Steinbook} that the spectral density of $Y$ under $P_j$ is
		\begin{equation*}
			\overline{f}^{\delta}_j(u) = \frac{1}{\delta} \sum_{\ell = -\infty}^{\infty} f_j\left(\frac{u + 2 \pi \ell}{\delta}\right) \quad \mbox{for } u \in (- \pi, \pi], \quad j = 0, 1\;,
		\end{equation*}
		where $f_j$ is the spectral density defined by \eqref{eq:spden} corresponding to $P_j$. For $j = 0, 1$, $f_j(u) \sim C \sigma_0^2 \phi_0^{2 \nu} u^{-2 \nu -1}$ as $u \to \infty$. From \citet[p.80, (17)]{Steinbook}, we obtain
		\begin{equation*}
			E_j e_{\infty, j}^2 \sim 2 \pi C \sigma_0^2 \phi_0^{2 \nu} \delta^{2 \nu} \exp\left( \frac{1}{2 \pi} \int_{- \pi}^{\pi} \log \left(\Sigma_{\ell = -\infty}^{\infty}|u + 2 \pi \ell|^{-2 \nu -1} \right) du\right),
		\end{equation*}
		which implies \eqref{eq:inferror} with 
		\[
		C_0 = 2 \pi C \sigma_0^2 \phi_0^{2 \nu} \exp\left( \frac{1}{2 \pi} \int_{- \pi}^{\pi} \log \left(\Sigma_{\ell = -\infty}^{\infty}|u + 2 \pi \ell|^{-2 \nu -1} \right) du\right)\;.
		\]
		
		Turning to \eqref{eq:differror}, we know from \citet[p.142, (15)]{Baxter62} that
		\begin{equation}\label{eq:keyiden}
			E_j(e_{i,j} - e_{\infty,j})^2 = E_j e_{\infty, j}^2 \, E_j e_{i,j}^2 \sum_{m = i}^{\infty}|\phi_{m,j}(0)|^2,
		\end{equation}
		where $\phi_{m,j}(\cdot)$ are the Szeg\"{o} polynomials associated with the spectral $\overline{f}^{\delta}_j$ \citep[see Section~2.1 of][for background]{GS58}. Note that $E_j e_{\infty, j}^2 \sim C_0 \delta^{2 \nu}$ and $E_j e_{i,j}^2 \le E_j e_{1,j}^2 \to 0$ as $\delta \to 0$. It will be sufficient to establish 
		\begin{equation}
			\label{eq:momentphi}
			\sum_{m = 0}^{\infty} m^{\kappa} |\phi_{m,j}(0)| \le D_{\kappa} \quad \mbox{for some } D_{\kappa} > 0,
		\end{equation}
		in which case the identity \eqref{eq:keyiden} will imply \eqref{eq:differror}. The key observation of \cite{Baxter62} (Theorem 2.3) is that \eqref{eq:momentphi} holds if 
		the $\kappa^{th}$ moment of the Fourier coefficients associated with $\overline{f}^{\delta}_j$ is bounded from above by $D'_{\kappa}$ for some $D'_{\kappa} > 0$, i.e.
		\begin{equation*}
			\sum_{m = 0}^{\infty} m^{\kappa} |c_{m,j}| < D_{\kappa}', \quad \mbox{where } c_{m,j}: = \frac{1}{2 \pi} \int_{- \pi}^{\pi} \overline{f}^{\delta}_j(u) e^{-inu}du.
		\end{equation*}
		A sufficient condition for the latter to hold is that the $\kappa^{th}$ derivative of $\overline{f}^{\delta}_j$ is integrable, and 
		$\int_{-\pi}^{\pi} \left| \frac{d^{\kappa}}{du^{\kappa}} \overline{f}^{\delta}_j(u) \right| du \le D''_{\kappa}$,
		for some $D''_{\kappa} > 0$ which does not depend on $\delta < 1$. Breaking the sum of $\overline{f}^{\delta}_j$ according to $\ell = 0$ and $\ell \ne 0$ produces
		\begin{equation}\label{eq:momentest}
			\int_{-\pi}^{\pi} \left| \frac{d^{\kappa}}{du^{\kappa}} \overline{f}^{\delta}_j(u) \right| du
			\le A \int_{-\infty}^{\infty} (1 + u^{2\nu + 1 + \kappa})^{-1} du + A' \delta^{2 \nu} \sum_{\ell \ne 0} l^{-2\nu-1-\kappa}\;,
		\end{equation}
		where $A, A' > 0$ are numerical constants. Hence, the right hand side of \eqref{eq:momentest} is bounded by $D''_{\kappa} = A \int_{-\infty}^{\infty}(1 + u^{2\nu + 1 + \kappa})^{-1} du + A'  \sum_{\ell \ne 0} l^{-2\nu-1-\kappa}$, which depends only on $\kappa$.
	\end{proof}
	
	\begin{proof}[Proof of Theorem~\ref{thm:main2}]
		We fix $\kappa = 2/ \epsilon$, and use $C^{(1)}_{\kappa}, C^{(2)}_{\kappa}, \ldots$ to denote constants depending only on $\kappa$.
		Note that $E_0e_{i,0}^2 = E_0 e_{\infty, 0}^2 + E_0(e_{\infty,0} - e_{i,0})^2$, since $e_{\infty,0}$ and $e_{\infty,0} - e_{i,0}$ are independent under $P_0$.
		By Proposition \ref{lem:key}, we get $E_0 e_{i,0}^2 \sim C_0 \delta^{2 \nu} (1 + C^{(1)}_{\kappa} i^{-\kappa})$. Similarly, $E_1 \tilde{e}^2_{i,1} = E_1 e_{\infty,1}^2 + E_1(\tilde{e}_{i,1} - e_{\infty,1})^2 \sim C_0 \delta^{2 \nu} (1 + C^{(2)}_{\kappa} \min(i, n^{\epsilon})^{-\kappa})$,
		because $\tilde{e}_{i,1}$ is realized as $y_0 - E_1\left(y_0 \given S_{(i-1)}[k]\right)$ with $k = \min(i, n^{\epsilon})$. Therefore,
		\begin{align*}
			\sum_{i = 1}^n \left(\frac{E_1 \tilde{e}^2_{i,1}}{E_0 e_{i,0}^2} - 1 \right)^2 & \le \sum_{i \ge n^{\epsilon}} \left(\frac{ C^{(3)}_{\kappa} n^{-\epsilon \kappa}}{1 + C^{(1)}_{\kappa} i^{-\kappa}}\right)^2 + \sum_{i = 1}^n \left( \frac{ C^{(4)}_{\kappa} i^{-\kappa}}{1 + C^{(1)}_{\kappa} i^{-\kappa}} \right)^2 \\
			& \le \frac{C^{(5)}_{\kappa}}{n^3} + C^{(6)}_{\kappa} \sum_{i = 1}^n i^{-2 \kappa} \le C_{\kappa}^{(7)}.
		\end{align*}
		Moreover, we have
		\begin{equation*}
			\frac{E_1(\tilde{e}_{i, 1} - e_{i, 0})^2}{E_0e_{i, 0}^2} \le 3 \frac{E_1(\tilde{e}_{i,1} - e_{\infty,1})^2}{E_0e_{i, 0}^2}
			+3 \frac{E_1(e_{i,1} - e_{\infty,1})^2}{E_0e_{i, 0}^2} + 3 \frac{E_1(e_{i,1} - e_{i,0})^2}{E_1 e_{i,1}^2} \frac{E_1 e_{i,1}^2}{E_0e_{i, 0}^2}.
		\end{equation*}
		By the same argument as above, for the first two terms:
		\begin{equation*}
			\sum_{i = 1}^n \frac{E_1(\tilde{e}_{i,1} - e_{\infty,1})^2}{E_0e_{i, 0}^2} < C^{(8)}_{\kappa} \quad \mbox{and} \quad
			\sum_{i=1}^n \frac{E_1(e_{i,1} - e_{\infty,1})^2}{E_0e_{i, 0}^2}  < C^{(9)}_{\kappa}.
		\end{equation*}
		For the last term,
		\begin{equation*}
			\sum_{i=1}^n \frac{E_1(e_{i,1} - e_{i,0})^2}{E_1 e_{i,1}^2} \frac{E_1 e_{i,1}^2}{E_0e_{i, 0}^2} 
			\le C^{(10)}_{\kappa} \sum_{i=1}^n \frac{E_1(e_{i,1} - e_{i,0})^2}{E_1 e_{i,1}^2},
		\end{equation*}
		which converges because of (\ref{eq:miscomp}). This establishes \eqref{eq: thm1_cond3} and, hence, \eqref{eq: thm1_vecch_MLE_asym_distr} follows.
	\end{proof}
	
	\section{Simulations
	}\label{subsec: sim_1}
	
	Based on \eqref{eq: cond_to_proveA}~and~\eqref{eq: cond_to_proveB} provided in Theorem~\ref{thm:cond_matern}, Theorem~\ref{thm:main2} has proved that
	\begin{equation}\label{eq: sim1_value}
		c_n(\phi_1, \phi_0, k) = \frac{1}{\sqrt{n}} \left[ 
		\sum_{i = 1}^{n} \frac{\tilde{e}_{i, 1}^2}{E_1\tilde{e}_{i, 1}^2} - \sum_{i = 1}^{n} \frac{e_{i, 0}^2}{E_0e_{i, 0}^2}\right] = o(1)\; 
	\end{equation}
	when $k = n^\epsilon$ for $\epsilon \in (0, 1)$. The equation \eqref{eq: sim1_value} induces the critical condition \eqref{eq: cond_to_prove}, resulting in the convergence in law in \eqref{eq: thm1_vecch_MLE_asym_distr}. Looking into the more challenging case $k = \mathcal{O}(\log(n))$, we extend the discussion in Theorem~\ref{thm:main2} via investigating the behaviour of $c_n(\phi_1, \phi_0, k)$ in \eqref{eq: sim1_value} for a sequence of datasets with increasing sample sizes. 
	{Our experiments involve two data generation schemes. The first scheme considers} the study domains $\mathcal{D}_1 = [0,1]$ with $n$ observations on the grid $\chi_1 = \{i/(n - 1): 0 \leq i \leq n-1\}$ and $\mathcal{D}_2 = [0, 1]^2$ with $n = n_s^2$ observations on the grid $\chi_2 = \{ (i/(n_s - 1), j/(n_s - 1)) : 0 \leq i, \leq n_s-1, 0 \leq j \leq n_s -1\}$. 
	{With this scheme, we generate a sequence of datasets with increasing sample size on increasingly finer grids on the study domains. The second scheme generates data on a ``disturbed grid''. On $\mathcal{D}_1 = [0,1]$ with $n$ observations, $\chi_1$ comprises locations randomly sampled by $\mbox{N}(i/(n+2), 0.15/(n+2))$ for $i = 1, \ldots, n$. On $\mathcal{D}_2 = [0, 1]^2$ with $n = n_s^2$ observations, locations in $\chi_2$ are generated by $\{N(i/(n_s + 2), 0.15/(n_s+2)), N(i/(n_s + 2), 0.15/(n_s+2))\}$ for $i, j = 1, \ldots, n_s$. With this scheme, we first generate simulations with the largest sample size, and then randomly select successively larger subsets from the same dataset to examine the tendency of $c_n(\phi_1, \phi_0, k)$ with an increasing $n$. The first scheme matches the setup of our proofs in the preceding sections, and the second scheme serves as a more directly informative regime for simulation studies about asymptotics.} 
	{In practice, estimation using Vecchia's approximation (\ref{eq: vecchia_approx}) is complicated by the fixed ordering of locations. \cite{guinness2018permutation} has provided excellent practical insights into this issue that can considerably improve finite sample behaviour in certain settings. In this study, we test two different orderings of locations, maximin ordering and sorted
		coordinate ordering. The sorted
		coordinate ordering} orders locations on $\chi_2$ first based on the second coordinate and then break ties based on the associated first coordinate. We take $S_{(i)}[k]$ as the 
	{at most} $k$ nearest neighbors of $y_{i+1}$. 
	In both studies on $\mathcal{D}_1$ and $\mathcal{D}_2$, we fix $\sigma^2=1.0$ and consider $5$ different smoothness values $\nu \in \{0.25, 0.5, 1.0, 1.5, 2\}$. We choose different decay parameters $\phi_0$ for different $\nu$ so that $K_{\theta}(h)=0.05$ when $h=0.2$ and $0.5$ for the study on $\mathcal{D}_1$ and $\mathcal{D}_2$, respectively. 
	
	For each fixed value of $\theta=\{\sigma^2,\phi,\nu\}$, we generate $100$ datasets with $Y_n$ being the realization from $y(s)\sim GP(0, K_{\theta}(\cdot))$ and calculate $c_n(\phi_1, \phi_0, k)$ with $k$ being the closest integer to 
	{$3\log(n)$}
	{, and $\phi_1 = 1.2 \phi_0$ and $1.1 \phi_0$ for $\mathcal{D}_1$ and $\mathcal{D}_2$, respectively.} Then, we record the mean and standard deviation of the $100$ values of $c_n(\phi_1, \phi_0, k)$. We repeat this process for different values of $n$ ranging from $2^6 = 64$ to 
	{$2^{12} = 4096$} in the study on $\mathcal{D}_1$. The study on $\mathcal{D}_2$ follows the study on $\mathcal{D}_1$ with $n_s$ ranging from 
	{$9$ to $81$}. The code for this simulation study is available on \url{https://github.com/LuZhangstat/vecchia_consistency}. 
	{Figure~\ref{fig: sim1_1d_a}~\&~\ref{fig: sim1_1d_b} summarize the study results on $\mathcal{D}_1$ under the two data generation schemes. Each figure presents 10 different graphs, one for each value of $\nu$ and each ordering, showing the mean and standard deviation of $c_n(\phi_1, \phi_0, k)$ for different values of $n$.}
	\begin{figure}[ht]
		\begin{subfigure}[b]{\textwidth}
			\centering
			\includegraphics[width = 0.9\linewidth]{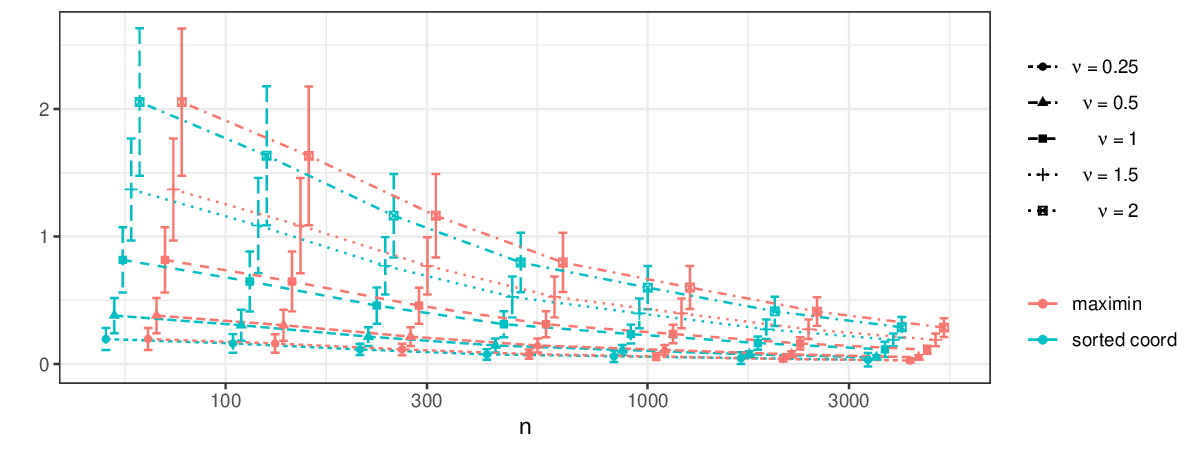}
			\caption{Increasingly finer grids (Data generation scheme 1)}
			\label{fig: sim1_1d_a}
		\end{subfigure}
		\begin{subfigure}[b]{\textwidth}
			\centering
			\includegraphics[width = 0.9\linewidth]{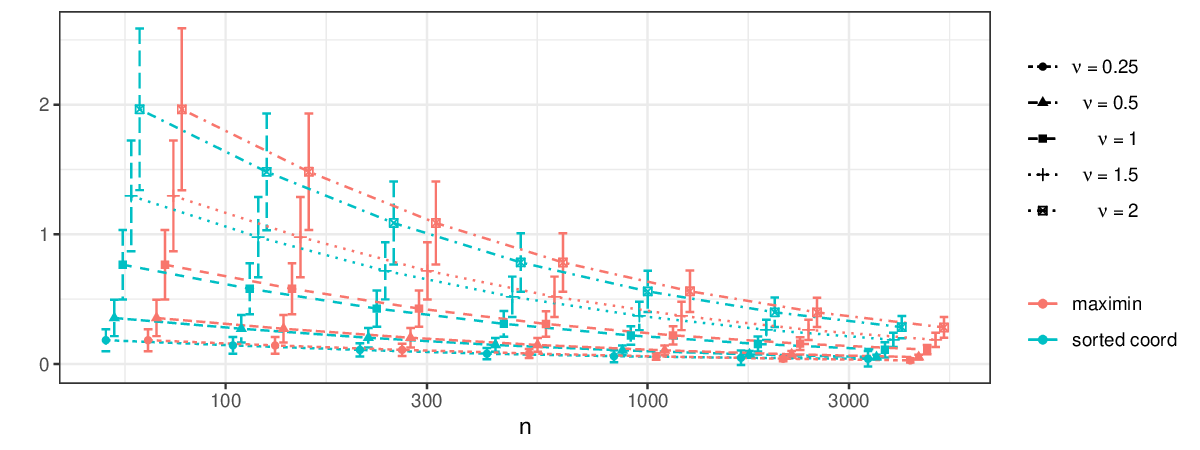}
			\caption{Successively increasing datasets (Data generation scheme 2)}
			\label{fig: sim1_1d_b}
		\end{subfigure}
		\caption{The mean of $c_n(\phi_1, \phi_0, k)$ of $100$ simulations on $\mathcal{D}_1 = [0, 1]$. The error bars represent one standard deviation. The sample size $n$ take on values in $64, 128, 256, 512, 1024, 2048$ 
			{and $4096$.} 
			{The graphs in red and blue show the results using maximin ordering and sorted coordinate ordering, respectively.}
			\label{fig: sim1}}
	\end{figure}
	
	The value of $c_n(\phi_1, \phi_0, k)$, as seen in Figure~\ref{fig: sim1}, decreases rapidly as the sample size increases, supporting the main conclusion in Theorem~\ref{thm:main2}. 
	{We do not observe a strong impact of ordering and data generation scheme on the results.}
	The corresponding graphs for the study on $\mathcal{D}_2$ are presented in Figure~\ref{fig: sim1_d2}. These graphs also reveal decreasing trends, but with more gentle slopes as compared to Figure~\ref{fig: sim1}. 
	{The results under the second data generation scheme are slightly better than those under the first scheme. When the smoothness $\nu$ is small, the standard deviation decreases faster with maximin ordering than with sorted coordinate ordering. Meanwhile, the standard deviation doesn't decrease significantly with the increase of $n$ when $\nu$ is large. 
	}
	\begin{figure}[ht]
		\begin{subfigure}[b]{\textwidth}
			\centering
			\includegraphics[width = 0.9\linewidth]{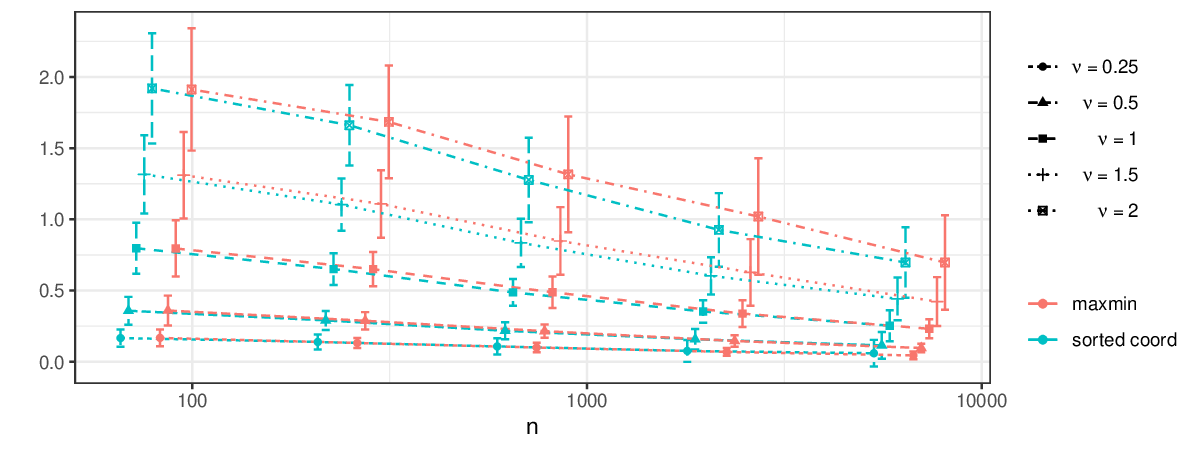}
			\caption{Increasingly finer grids (Data generation scheme 1)}
			\label{fig: sim1_d2_a}
		\end{subfigure}
		\begin{subfigure}[b]{\textwidth}
			\centering
			\includegraphics[width = 0.9\linewidth]{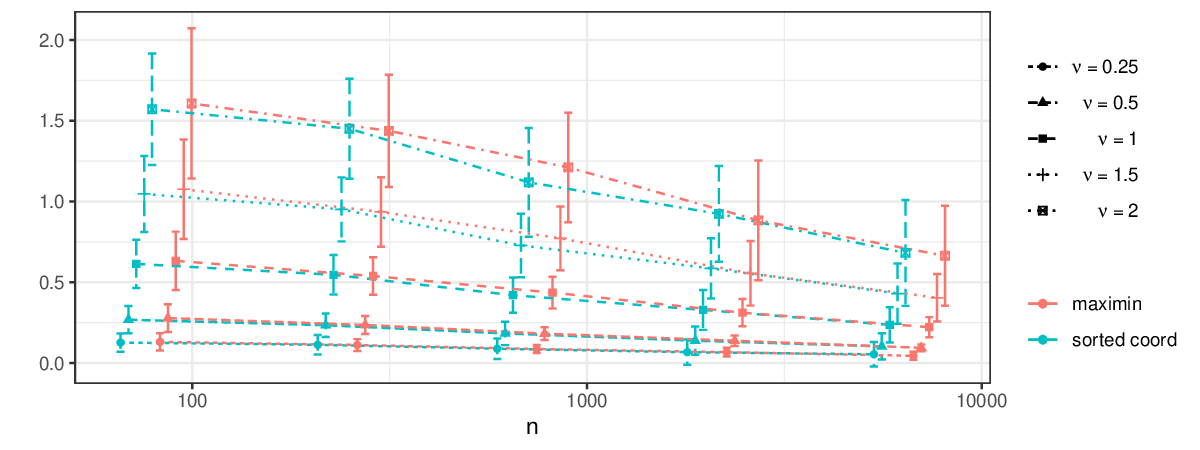}
			\caption{Successively increasing datasets (Data generation scheme 2)}
			\label{fig: sim1_d2_b}
		\end{subfigure}
		\caption{The mean of $c_n(\phi_1, \phi_0, k)$ of $100$ simulations on $\mathcal{D}_2 = [0, 1]^2$. The error bars represent one standard deviation. The sample size $n$ take on values in 
			{$81, 256, 729, 2209$ and $6561$.} 
			{The graphs in red and blue show the results using maximin ordering and sorted coordinate ordering, respectively. }
			\label{fig: sim1_d2}}
	\end{figure}
	{To explore further, we reproduce the study on $\mathcal{D}_2$ with $k$ being the closet integer to $\sqrt{n}$, and we illustrate the results in Figure~\ref{fig: sim1_d2_sqrt}. We observe that the standard deviation decreases rapidly as $n$ increases in all cases.}
	\begin{figure}[ht]
		\begin{subfigure}[b]{\textwidth}
			\centering
			\includegraphics[width = 0.9\linewidth]{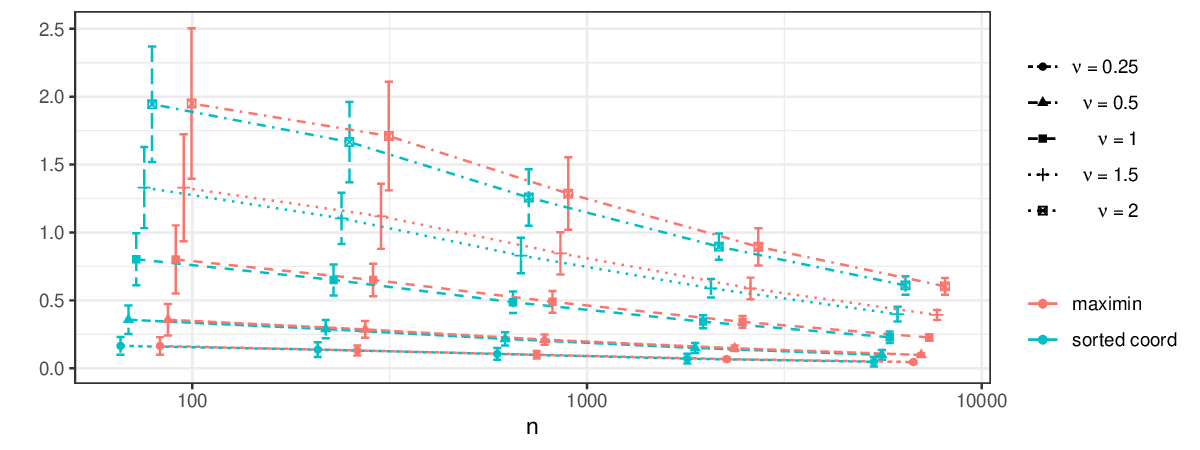}
			\caption{Increasingly finer grids (Data generation scheme 1)}
			\label{fig: sim1_d2_sqrt_a}
		\end{subfigure}
		\begin{subfigure}[b]{\textwidth}
			\centering
			\includegraphics[width = 0.9\linewidth]{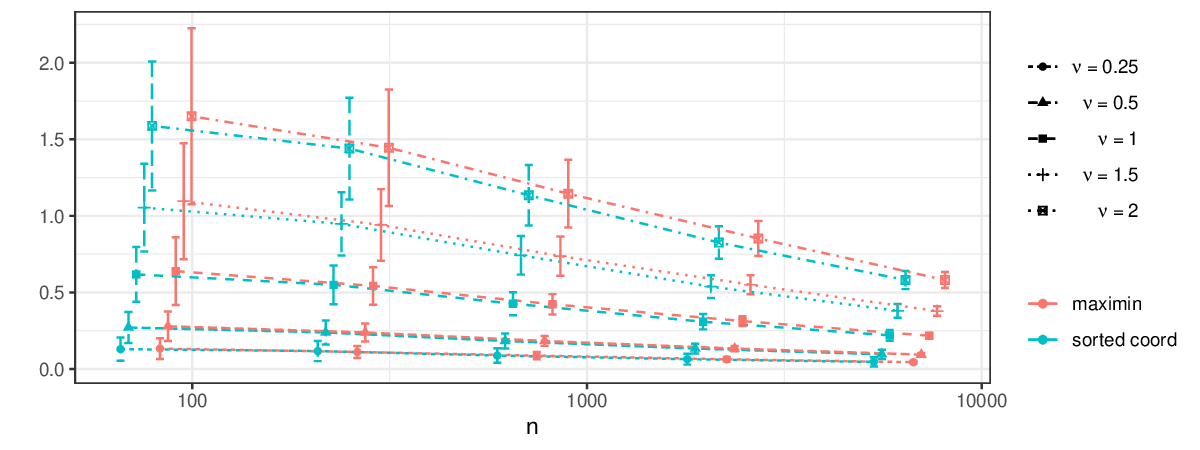}
			\caption{Successively increasing datasets (Data generation scheme 2)}
			\label{fig: sim1_d2_sqrt_b}
		\end{subfigure}
		\caption{The mean of $c_n(\phi_1, \phi_0, k)$ of $100$ simulations on $\mathcal{D}_2 = [0, 1]^2$. The error bars represent one standard deviation. The sample size $n$ take on values in 
			{$81, 256, 729, 2209$ and $6561$.} 
			{The graphs in red and blue show the results using maximin ordering and sorted coordinate ordering, respectively. }
			\label{fig: sim1_d2_sqrt}}
	\end{figure}
	
	We have also seen, 
	from the proof in Theorem~\ref{thm:cond_matern}, that $c_n(k, \phi_1, \phi_0) = \sqrt{n}(\hat{\sigma}^2_{n, vecch} / \sigma_1^2 - \hat{\sigma}_{0, n}^2 / \sigma_0^2)$ where $\hat{\sigma}_{0, n}^2 = \mbox{argmax}_{\sigma^2}\{ p(y; \phi_0, \sigma^2),\; \sigma^2 \in \mathbb{R}^{+}\} \; $ is the maximum likelihood estimator from (\ref{eq: full_likelihood})  when fixing $\phi_1 = \phi_0$. Hence, $c_n(k, \phi_1, \phi_0)$ also measures the discrepancy between $\hat{\sigma}^2_{n, vecch} / \sigma_1^2$ and $\hat{\sigma}_{0, n}^2 / \sigma_0^2$, and the decreasing trend of $c_n(\phi_0, \phi_1, k)$ indicates that the inference based on Vecchia's approximation approaches the inference based on the full likelihood as sample size increases. This phenomenon reveals that Vecchia's approximation is still efficient when the neighbor size $k$ is substantially smaller than the sample size.

	\section{Conclusions and Future work}\label{sec: conclusion}
	
	We have developed insights into inference based on GP likelihood approximations by \cite{ve88} under fixed domain asymptotics for geostatistical data analysis. We have formally established the sufficient conditions for such approximations to have the same asymptotic efficiency as a full GP likelihood in estimating parameters in Mat\'ern covariance function. The insights obtained here will enhance our understanding of identifiability of process parameters and can also be useful for developing priors for the microergodic parameters in Bayesian modeling. The results derived here will also offer insights into formally establishing posterior consistency of process parameters for a number of Bayesian models that have emerged from (\ref{eq: vecchia_approx}) \citep{datta16a, datta16b, katzfuss2021, pbf2022}.
	
	We anticipate the current manuscript to generate further research in variants of geostatistical models. For example, it is conceivable that these results will lead to asymptotic investigations of covariance-tapered models \citep[see, e.g.,][]{wang2011fixed} and in adapting some results, such as Theorems~2~and~3 in \cite{KS13} where $\phi$ is estimated, to the approximate likelihoods presented here. Another direction of research can lead to formal developments regarding the inferential consistency of the ``nugget'' or the variance of measurement error when the spatial process has a discontinuity at $0$ arising white noise \citep{tzb2021}. Finally, there is scope to specifically investigate fixed domain inference for other likelihood approximations that extend or generalize (\ref{eq: vecchia_approx}) \citep[see, e.g.,][]{katzfuss2021, pbf2022}.

	%
	%
	
	\section*{Acknowledgements}
	
	The work of the first and third author was supported, in part, by funding from the National Science Foundation grants NSF/DMS 1916349 and NSF/IIS 1562303, and by the National Institute of Environmental Health Sciences (NIEHS) under grants R01ES030210 and 5R01ES027027.


	\bibliography{unique.bib}
	\bibliographystyle{plainnat}

\label{lastpage}

\end{document}